\documentclass[a4paper,10pt,twoside]{amsart}
\usepackage{a4wide}
\usepackage{amsmath}
\usepackage{amsthm}
\usepackage{amsfonts}
\usepackage{amssymb}
\usepackage{tikz}
\usepackage[initials,non-compressed-cites,bibtex-style]{amsrefs}
\usepackage{url}
\usetikzlibrary{arrows,positioning,decorations.text}

\theoremstyle{plain}
\newtheorem{theorem}{Theorem}[section]
\newtheorem{lemma}[theorem]{Lemma}
\newtheorem{corollary}[theorem]{Corollary}
\newtheorem{conjecture}[theorem]{Conjecture}

\theoremstyle{definition}
\newtheorem{definition}[theorem]{Definition}

\theoremstyle{remark}
\newtheorem*{remark}{Remark}

\def \t {\mathbf t}
\def \x {\mathbf x}
\def \v {\mathbf v}
\def \n {\mathbf n}
\def \m {\mathbf m}
\def \c {\mathbf c}
\def \e {\mathbf e}
\def \w {\mathbf w}
\def \R {\mathbf R}
\def \U {\mathbf U}
\def \V {\mathbf V}
\def \A {\mathbf A}
\def \M {\mathbf M}
\def \I {\mathbf I}

\begin{document}

\title[Ball-packability]{Apollonian Ball Packings and Stacked Polytopes}
\author{Hao Chen}
\address[H. Chen]{Freie Universit\"at Berlin, Institut f\"ur Mathematik}
\begin{abstract}
  We investigate in this paper the relation between Apollonian $d$-ball packings and stacked $(d+1)$-polytopes for dimension $d\ge 3$. For $d=3$, the relation is fully described: we prove that the $1$-skeleton of a stacked $4$-polytope is the tangency graph of an Apollonian $3$-ball packing if and only if no six $4$-cliques share a $3$-clique. For higher dimension, we have some partial results.
\end{abstract}
\thanks{
The author was supported by the Deutsche Forschungsgemeinschaft within the Research Training Group `Methods for Discrete Structures' (GRK 1408).
}
\maketitle

\section{Introduction}
A ball packing is a collection of balls with disjoint interiors.  A graph is said to be ball packable if it can be realized by the tangency relations of a ball packing.  The combinatorics of disk packings ($2$-dimensional ball packings) is well understood thanks to the Koebe--Andreev--Thurston's disk packing theorem, which asserts that every planar graph is disk packable. However, few is known about the combinatorics of ball packings in higher dimensions. 

In this paper we study the relation between Apollonian ball packings and stacked polytopes. An Apollonian ball packing is constructed from a Descartes configuration by repeatedly filling new balls into ``holes''. A stacked polytope is constructed from a simplex by repeatedly gluing new simplices onto facets. See Section~\ref{sse:apollonian} and~\ref{sse:stacked} respectively for formal descriptions.  There is a $1$-to-$1$ correspondence between $2$-dimensional Apollonian ball packings and $3$-dimensional stacked polytopes. Namely, a graph can be realised by the tangency relations of an Apollonian disk packing if and only if it is the $1$-skeleton of a stacked $3$-polytope.  However, this relation does not hold in higher dimensions.  

On one hand, the $1$-skeleton of a stacked polytope may not be realizable by the tangency relations of any Apollonian ball packing.  Our main result, proved in Section \ref{sec:stacked}, give a condition on stacked $4$-polytopes to restore the relation in this direction:
\begin{theorem}[Main result]\label{thm:main}
  The $1$-skeleton of a stacked $4$-polytope is $3$-ball packable if and only if it does not contain six $4$-cliques sharing a $3$-clique.
\end{theorem}
For higher dimensions, we propose Conjecture \ref{conj:higher} following the pattern of dimension~$2$ and~$3$.  On the other hand, the tangency graph of a ball packing may not be the $1$-skeleton of any stacked polytope. We prove in Corollary \ref{cor:hexletfree} that this is the case for a $3$-dimensional ball packing containing Soddy's hexlet, a special packing consisting of nine balls.  However, ball packings of dimension higher than $3$ do not have this problem, as we will prove in Theorem \ref{thm:not3d}.

The paper is organized as follows. In Section~\ref{sec:prelim}, we introduce the notions related to Apollonian ball packings and stacked polytopes. In Section~\ref{sec:join}, we construct ball packings for some graph joins. These constructions provide forbidden induced subgraphs for the tangency graphs of ball packings, which are helpful for the intuition, and some are useful in the proofs. The main result and related results are proved in Section~\ref{sec:stacked}. Finally, we discuss in Section~\ref{sec:discuss} about edge-tangent polytopes, an object closely related to ball packings.

\section{Definitions and preliminaries}\label{sec:prelim}
\subsection{Ball packings}

We work in the $d$-dimensional extended Euclidean space $\hat{\mathbb{R}}^d=\mathbb{R}^d\cup\{\infty\}$.  A \emph{$d$-ball} of curvature $\kappa$ means one of the following sets:
\begin{itemize}
  \item $\{\x\mid\lVert\x-\c\rVert\leq 1/\kappa\}$ if $\kappa>0$;
  \item $\{\x\mid\lVert\x-\c\rVert\geq -1/\kappa\}$ if $\kappa<0$;
  \item $\{\x\mid\langle\x,\hat\n\rangle\geq b\}\cup\{\infty\}$ if $\kappa=0$,
\end{itemize}
where $\lVert\cdot\rVert$ is the Euclidean norm, and $\langle\cdot,\cdot\rangle$ is the Euclidean inner product.  In the first two cases, the point $\c\in\mathbb{R}^d$ is called the \emph{center} of the ball.  In the last case, the unit vector $\hat\n$ is called the \emph{normal vector} of a half-space, and $b\in\mathbb{R}$. The boundary of a $d$-ball is a \emph{$(d-1)$-sphere}.  Two balls are tangent at a point $\t\in\hat{\mathbb{R}}^d$ if $\t$ is the only element of their intersection.  We call $\t$ the \emph{tangency point}, which can be the infinity point $\infty$ if it involves two balls of curvature $0$.  For a ball $S\subset\hat{\mathbb{R}}^d$, the \emph{curvature-center coordinates} is introduced by Lagarias, Mallows and Wilks in~\cite{lagarias2002} 
\[
\m(S)=
\begin{cases}
  (\kappa,\kappa\c)& \text{if } \kappa\neq 0;\\
  (0,\hat\n)& \text{if } \kappa=0.
\end{cases}
\]
Here, the term ``coordinate'' is an abuse of language, since the curvature-center coordinates do not uniquely determine a ball when $\kappa=0$.  A real global coordinate system would be the \emph{augmented curvature-center coordinates}, see~\cite{lagarias2002}.  However, the curvature-center coordinates are good enough for our use.
\begin{definition}\label{def:ballpacking}
A \emph{$d$-ball packing} is a collection of $d$-balls with disjoint interiors.
\end{definition}
For a ball packing $\mathcal{S}$, its \emph{tangency graph} $G(\mathcal{S})$ takes the balls as vertices and the tangency relations as the edges.  The tangency graph is invariant under M\"obius transformations and reflections.
\begin{definition}\label{def:ballpackable}
A graph $G$ is said to be \emph{$d$-ball packable} if there is a $d$-ball packing $\mathcal{S}$ whose tangency graph is isomorphic to $G$.  In this case, we say that $\mathcal{S}$ is a $d$-ball packing of $G$.
\end{definition}

\emph{Disk packing}, or $2$-ball packing, is well understood.
\begin{theorem}[Koebe--Andreev--Thurston theorem~\cites{koebe1936,thurston1979}]\label{thm:diskpack}
  Every connected simple planar graph is disk packable.  If the graph is a finite triangulated planar graph, then it has a unique disk packing up to M\"obius transformations.
\end{theorem}
Few is known about the combinatorics of ball packings in higher dimensions.  Some attempts of generalizing the disk packing theorem to higher dimensions include~\cites{kuperberg1994,benjamini2010,cooper1996,miller1997}.  Clearly, an induced subgraph of a $d$-ball packable graph is also $d$-ball packable.  In other words, the class of ball packable graphs is closed under the induced subgraph operation.  

Throughout this paper, ball packings are always in dimension $d$. The dimensions of other objects will vary correspondingly.

\subsection{Descartes configurations}
A \emph{Descartes configuration} in dimension $d$ is a $d$-ball packing consisting of $d+2$ pairwise tangent balls.  The tangency graph of a Descartes configuration is the complete graph on $d+2$ vertices. This is the basic element for the construction of many ball packings in this paper.  The following relation was first established for dimension $2$ 
by Ren\'e Descartes in a letter~\cite{descartes1643} to Princess Elizabeth of Bohemia, then generalized to dimension $3$ by Soddy in the form of a poem~\cite{soddy1936a}, and finally generalized to arbitrary dimension by Gossett~\cite{gossett1937}.
\begin{theorem}[Descartes--Soddy--Gossett Theorem]
  In dimension $d$, if $d+2$ balls $S_1,\cdots,S_{d+2}$ form a Descartes configuration, let $\kappa_i$ be the curvature of $S_i$ \textup{(}$1\leq i\leq d+2$\textup{)}, then
  \begin{equation}\label{eq:soddy}
    \sum_{i=1}^{d+2}\kappa_i^2=\frac{1}{d}\Big(\sum_{i=1}^{d+2}\kappa_i\Big)^2
  \end{equation}
\end{theorem}
Equivalently, we have $\mathbf{K}^\intercal\mathbf{Q}_d\mathbf{K}=0$,
where $\mathbf{K}=(\kappa_1,\cdots,\kappa_{d+2})^\intercal$ is the vector of curvatures, and $\mathbf{Q}_d:=\I-\frac{1}{d}\e\e^\intercal$ is a square matrix of size $d+2$, where $\e$ is the all-one column vector, and $\I$ is the identity matrix.  A more generalized relation on the curvature-center coordinates was proved in~\cite{lagarias2002}:
\begin{theorem}[Generalized Descartes--Soddy--Gossett Theorem]
  In dimension $d$, if $d+2$ balls $S_1,\cdots,S_{d+2}$ form a Descartes configuration, then
  \begin{equation}\label{eq:lagarias}
    \M^\intercal\mathbf{Q}_d\M=\begin{pmatrix} 0 & 0\\0 & 2\I \end{pmatrix}
  \end{equation}
  where $\M$ is the \emph{curvature-center matrix} of the configuration, whose $i$-th row is~$m(S_i)$.
\end{theorem}

Given a Descartes configuration $S_1,\cdots,S_{d+2}$, we can construct another Descartes configuration by replacing $S_1$ with an $S_{d+3}$, such that the curvatures $\kappa_1$ and $\kappa_{d+3}$ are the two roots of \eqref{eq:soddy} treating~$\kappa_1$ as unknown.  So we have the relation
\begin{equation}\label{eq:curvature}
  \kappa_1+\kappa_{d+3}=\frac{2}{d-1}\sum_{i=2}^{d+2}\kappa_i
\end{equation}
We see from \eqref{eq:lagarias} that
the same relation holds for all the entries in the curvature-center coordinates,
\begin{equation}\label{eq:curvcenter}
  \m(S_1)+\m(S_{d+3})=\frac{2}{d-1}\sum_{i=2}^{d+2}\m(S_i)
\end{equation}
These equations are essential for the calculations in the paper.

By recursively replacing $S_i$ with a new ball $S_{i+d+2}$ in this way,
we obtain an infinite sequence of balls $S_1,S_2,\cdots$,
in which any $d+2$ consecutive balls form a Descartes configuration.
This is \emph{Coxeter's loxodromic sequences} of tangent balls~\cite{coxeter1968}.

\subsection{Apollonian cluster of balls}\label{sse:apollonian}
\begin{definition}
A collection of $d$-balls is said to be \emph{Apollonian} if it can be built from a Descartes configuration by repeatedly introducing, for $d+1$ pairwise tangent balls, a new ball that is tangent to all of them. 
\end{definition}
Please note that a newly added ball is allowed to touch more than $d+1$ balls, and may intersect some other balls. In the latter case, the result is not a packing.  For example, Coxeter's loxodromic sequence is Apollonian.  In this paper, we are interested in (finite) \emph{Apollonian ball packings}.

We reformulate the replacing operation described before \eqref{eq:curvature} by inversions.  Given a Descartes configuration $\mathcal{S}=\{S_1,\cdots,S_{d+2}\}$, let $R_i$ be the inversion in the sphere that orthogonally intersects the boundary of $S_j$ for all $1\leq j\neq i\leq d+2$, then $R_i\mathcal{S}$ forms a new Descartes configuration, which keeps every ball of $\mathcal{S}$, except that $S_i$ is replaced by $R_iS_i$.  With this point of view, a Coxeter's sequence can be obtained from an initial Descartes configuration $\mathcal{S}_0$ by recursively constructing a sequence of Descartes configurations by $\mathcal{S}_{n+1}=R_{j+1}\mathcal{S}_n$ where $j\equiv n\pmod{d+2}$, then taking the union.

The group $W$ generated by $\{R_1,\dots,R_{d+2}\}$ is called the \emph{Apollonian group}.  The union of the orbits $\cup_{S\in\mathcal{S}_0}WS$ is called the \emph{Apollonian cluster} (of balls)~\cite{graham2006}.  The Apollonian cluster is an infinite ball packing in dimensions two~\cite{graham2005} and three~\cite{boyd1973}. That is, the interiors of any two balls in the cluster are either identical or disjoint.  This is unfortunately not true for higher dimensions.  Our main object of study, Apollonian ball packings, can be seen as special subsets of the Apollonian cluster.

Define
\[
\R_i:=\I+\frac{2}{d-1}\e_i\e^\intercal-\frac{2d}{d-1}\e_i\e_i^\intercal
\]
where $\e_i$ is a $(d+2)$-vector whose entries are $0$ except for the $i$-th entry being~$1$.  So~$\R_i$ coincide with the identity matrix at all rows except for the $i$-th row, whose diagonal entry is $-1$ and the off-diagonal entries are $2/(d-1)$.  One then verifies that~$\R_i$ induces a representation of the Apollonian group.  In fact, if $\M$ is the curvature-center matrix of a Descartes configuration $\mathcal{S}$, then $\R_i\M$ is the curvature-center matrix of $R_i\mathcal{S}$.

\subsection{Stacked polytopes}\label{sse:stacked}

For a simplicial polytope, a \emph{stacking operation} glues a new simplex onto a facet.

\begin{definition}
  A simplicial $d$-polytope is \emph{stacked} if it can be iteratively constructed from a $d$-simplex by a sequence of \emph{stacking operations}.
\end{definition}

We call the $1$-skeleton of a polytope $\mathcal{P}$ the \emph{graph} of $\mathcal{P}$, denoted by $G(\mathcal{P})$.  For example, the graph of a $d$-simplex is the complete graph on $d+1$ vertices.  The graph of a stacked $d$-polytope is a \emph{$d$-tree}, that is, a chordal graph whose maximal cliques are of a same size $d+1$.  Inversely, 
\begin{theorem}[Kleinschmidt~\cite{kleinschmidt1976}]\label{thm:ktree}
  A $d$-tree is the graph of a stacked $d$-polytope if and only if there is no three $(d+1)$-cliques sharing $d$ vertices.
\end{theorem}
A $d$-tree satisfying this condition will be called \emph{stacked $d$-polytopal graph}. 

A simplicial $d$-polytope $\mathcal{P}$ is stacked if and only if it admits a triangulation~$\mathcal{T}$ with only interior faces of dimension $(d-1)$. For $d\geq 3$, this triangulation is unique, whose simplices correspond to the maximal cliques of $G(\mathcal{P})$.  This implies that stacked polytopes are uniquely determined by their graph (i.e.~stacked polytopes with isomorphic graphs are combinatorially equivalent).  The \emph{dual tree}~\cite{gonska2011} of $\mathcal{P}$ takes the simplices of $\mathcal{T}$ as vertices, and connect two vertices if the corresponding simplices share a $(d-1)$-face.

The following correspondence between Apollonian $2$-ball packings and stacked $3$-polytopes can be easily seen from Theorem~\ref{thm:diskpack} by comparing the construction processes:
\begin{theorem}
If a disk packing is Apollonian, then its tangency graph is stacked $3$-polytopal.  If a graph is stacked $3$-polytopal, then it is disk packable, and its disk packing is Apollonian and unique up to M\"obius transformations. 
\end{theorem}

The relation between $3$-tree, stacked $3$-polytope and Apollonian $2$-ball packing can be illustrated as follows:
  \begin{center}
    \begin{tikzpicture}[node distance=3cm]
      \node (tree) {$3$-tree};
      \node (polytope) [below left=of tree, xshift=1cm] {stacked $3$-polytope};
      \node (packing) [below right=of tree, xshift=-1cm] {Apollonian $2$-ball packing};
      \draw [->>] (tree) to node [sloped,above] {no three $4$-cliques} node [sloped,below] {sharing a $3$-clique} (polytope);
      \draw [->>] (tree) to node [sloped,above] {no three $4$-cliques} node [sloped,below] {sharing a $3$-clique} (packing);
      \draw [<->] (polytope) to (packing);
    \end{tikzpicture}
  \end{center}
where the double-headed arrow $A\twoheadrightarrow B$ means that every instance of $B$ corresponds to an instance of $A$ satisfying the given condition.

\section{Ball-packability of graph joins}\label{sec:join}
\subsection*{Notations}

We use $G_n$ to denote any graph on $n$ vertices, and use
\begin{description}
  \item[$P_n$] for the path on $n$ vertices (therefore of length $n-1$); 
  \item[$C_n$] for the cycle on $n$ vertices; 
  \item[$K_n$] for the complete graph on $n$ vertices;
  \item[$\bar K_n$] for the empty graph on $n$ vertices;
  \item[$\lozenge_d$] for the $1$-skeleton of the $d$-dimensional orthoplex\footnote{also called ``cross polytope''};
\end{description}

The \emph{join} of two graphs $G$ and $H$, denoted by $G\star H$, is the graph obtained by connecting every vertex of $G$ to every vertex of $H$.  Most of the graphs in this section will be expressed in term of graph joins.  Notably, we have $\lozenge_d=\underbrace{\bar K_2\star\dots\star\bar K_2}_d$.  

\subsection{Graphs in the form of $K_d\star P_m$}
The following theorem reformulates a result of Wilker~\cite{wilker1972}.  A proof was sketched in~\cite{boyd1973}.  Here we present a very elementary proof, suitable for our further generalization.

\begin{theorem}\label{thm:pk}
  Let $d\geq 2$ and $m\geq 0$. A graph in the form of
  \begin{enumerate}
      \renewcommand{\theenumi}{\rm (\roman{enumi})}
      \renewcommand{\labelenumi}{\theenumi}
    \item \label{forb:d2} $K_2\star P_m$ is $2$-ball packable for any $m$;
    \item \label{forb:m4} $K_d\star P_m$ is $d$-ball packable if $m\leq 4$;
    \item \label{forb:m6} $K_d\star P_m$ is \emph{not} $d$-ball packable if $m\geq 6$;
    \item \label{forb:m5} $K_d\star P_5$ is $d$-ball packable if and only if $2\leq d\leq 4$;
  \end{enumerate}
\end{theorem}

\begin{proof}
 ~\ref{forb:d2} is trivial.

  For dimension $d>2$, we construct a ball packing for the $(d+1)$-simplex $K_{d+2}=K_d\star P_2$ as follows. The two vertices of $P_2$ are represented by two disjoint half-spaces $\mathsf{A}$ and $\mathsf{B}$ at distance $2$ apart, and the $d$ vertices of $K_d$ are represented by $d$ pairwise tangent unit balls touching both $\mathsf{A}$ and $\mathsf{B}$.  Figure~\ref{pic:C6K3} shows the situation for $d=3$, where red balls represent vertices of $K_3$. This is the unique packing of $K_{d+2}$ up to M\"obius transformations.

  Let $\mathsf{S}$ be the $(d-2)$-sphere decided by the centers of the unit balls.  The idea of the proof is the following. Starting from $K_d\star P_2$, we construct the ball packing of $K_d\star P_m$ by appending new balls to the path, touching all the $d$ unit balls representing $K_d$. These new balls must center on a straight line passing through the center of $\mathsf{S}$ perpendicular to the hyperplane containing $\mathsf{S}$.  The construction fails when the sum of the diameters exceeds $2$.

  \begin{figure}[htb]
    \centering
    \includegraphics[width=\textwidth]{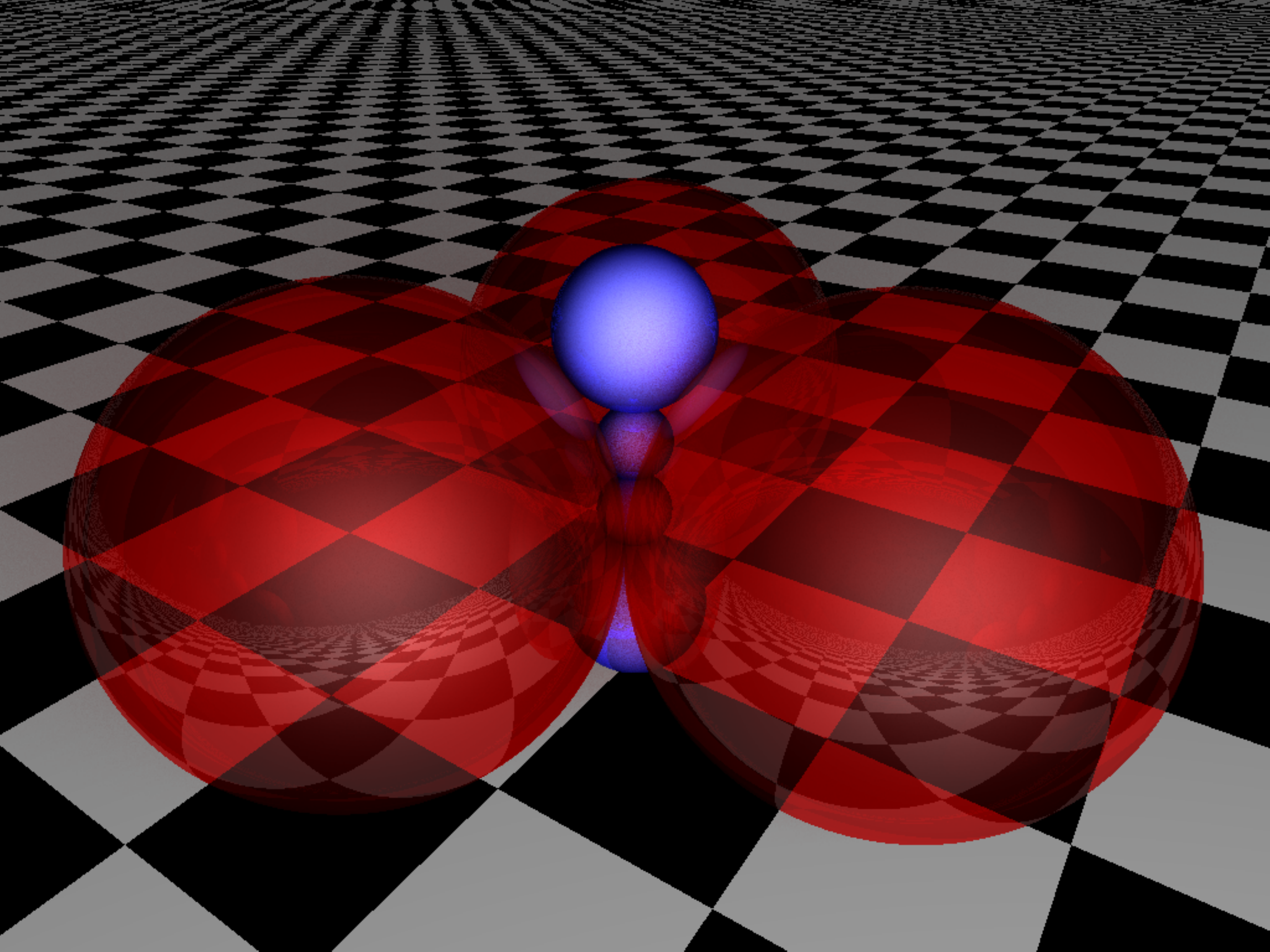}
    \caption{\label{pic:C6K3} An attempt of constructing the ball packing of $K_3\star P_6$ results in $K_3\star C_6$. Refering to the proof of Theorem~\ref{thm:pk}, the red balls correspond to vertices in $K_3$, the blue balls are labeled by $\mathsf{C}$, $\mathsf{E}$, $\mathsf{F}$, $\mathsf{D}$ from bottom to top. The upper half-space $\mathsf{B}$ is not shown. The image is rendered by POV-Ray.}
  \end{figure}

  As a first step, we construct $K_d\star P_3$ by adding a new ball $\mathsf{C}$ tangent to $\mathsf{A}$.  By \eqref{eq:curvature}, the diameter of $\mathsf{C}$ is $2/\kappa_\mathsf{C}=(d-1)/d<1$.  Since $\mathsf{C}$ is disjoint from $\mathsf{B}$, this step succeeded.  Then we add the ball $\mathsf{D}$ tangent to $\mathsf{B}$.  It has the same diameter as $\mathsf{C}$ by symmetry, and they sum up to $2(d-1)/d<2$.  So the construction of $K_d\star P_4$ succeeded, which proves~\ref{forb:m4}.

  We now add a ball $\mathsf{E}$ tangent to $\mathsf{C}$.  Still by \eqref{eq:curvature}, the diameter of $\mathsf{E}$ is \[\frac{2}{\kappa_\mathsf{E}}=\frac{(d-1)^2}{d(d+1)}\] If we sum up the diameters of $\mathsf{C}$, $\mathsf{D}$ and $\mathsf{E}$, we get
  \begin{equation}\label{eqn:5ball}
    2\frac{d-1}{d}+\frac{(d-1)^2}{d(d+1)}=\frac{3d^2-2d-1}{d(d+1)}
  \end{equation}
  which is smaller then $2$ if and only if $d\leq 4$.  Therefore the construction fails unless $2\leq d\leq 4$, which proves~\ref{forb:m5}.

  For $2\leq d\leq 4$, we continue to add a ball $\mathsf{F}$ tangent to $\mathsf{D}$.  It has the same diameter as $\mathsf{E}$.  If we sum up the diameters of $\mathsf{C}$, $\mathsf{D}$, $\mathsf{E}$ and $\mathsf{F}$, we get
  \begin{equation}\label{eqn:6ball}
    2\left(\frac{d-1}{d}+\frac{(d-1)^2}{d(d+1)}\right)=4\frac{d-1}{d+1}
  \end{equation}
  which is smaller then $2$ if and only if $d<3$, which proves~\ref{forb:m6}.
\end{proof}
\begin{remark}
  Figure~\ref{pic:C6K3} shows the attempt of constructing the ball packing of $K_3\star P_6$ but yields the ball packing of $K_3\star C_6$.  This packing is called \emph{Soddy's hexlet}~\cite{soddy1936b}.  It's an interesting configuration since the diameters of $\mathsf{C}$, $\mathsf{D}$, $\mathsf{E}$ and $\mathsf{F}$ sum up to exactly $2$.  This configuration is also studied by Maehara and Oshiro in~\cite{maehara2000}.
\end{remark}
\begin{remark}
  Let's point out the main differences between the situation in dimension $2$ and higher dimensions. For $d=2$, a Descartes configuration divides the space into $4$ disjoint regions, and the radius of a circle tangent to the two unit circles of $K_2$ can be arbitrarily small.  However, if $d>2$, the complement of a Descartes configuration is always connected, and the radius of a ball tangent to all the $d$ balls of $K_d$ is bounded away from $0$.  In fact, using the Descartes--Soddy--Gossett theorem, one verifies that the radius of such a ball is at least $\frac{d-2}{d+\sqrt{2d^2-2d}}$, which tends to $\frac{1}{1+\sqrt 2}$ as $d$ tends to infinity.
\end{remark}

\subsection{Graphs in the form of $K_n\star G_m$}

The following is a corollary of Theorem~\ref{thm:pk}.
\begin{corollary}\label{cor:gk}
  For $d=3$ or $4$, a graph in the form of $K_d\star G_6$ is not $d$-ball packable, with the exception of $K_3\star C_6$.  For $d\geq 5$, a graph in the form of $K_d\star G_5$ is not $d$-ball packable.
\end{corollary}

\begin{proof}
  For construction of $K_d\star G_m$, we just repeat the construction in the proof of Theorem~\ref{thm:pk}.  Since the centers of the balls of $G_m$ are situated on a straight line, $G_m$ can only be a path, a cycle $C_m$ or a disjoint union of paths (possibly empty).  The first possibility is ruled out by Theorem~\ref{thm:pk}.  The cycle is only possible when $d=3$ and $m=6$, in which case the ball packing of $K_3\star C_6$ is Soddy's hexlet.  If $G_m$ is a disjoint union of paths, we are forced to leave gaps between balls, but Theorem~\ref{thm:pk} says that there is no space for any gap. So the construction is not possible.
\end{proof}

We now study some other graphs with the form $K_n\star G_m$ using kissing configuration and spherical codes.  A \emph{$d$-kissing configuration} is a packing of unit $d$-balls all touching another unit ball. The \emph{$d$-kissing number} $k(d,1)$ (the reason for this notation will be clear later) is the maximum number of balls in a $d$-kissing configuration.  The kissing number is known to be $2$ for dimension~$1$, $6$ for dimension~$2$, $12$ for dimension~$3$~\cite{conway1999}, $24$ for dimension~$4$~\cite{musin2003}, $240$ for dimension~$8$ and $196560$ for dimension~$24$~\cite{odlyzko1979}.  We have immediately the following theorem.
\begin{theorem}
  A graph in the form of $K_3\star G$ is $d$-ball packable if and only if $G$ is the tangency graph of a $(d-1)$-kissing configuration.
\end{theorem}
To see this, just represent $K_3$ by one unit ball and two disjoint half-spaces at distance $2$ apart, then the other balls must form a $(d-1)$-kissing configuration.  For example, $K_3\star G_{13}$ is not $4$-ball packable, $K_3\star G_{25}$ is not $5$-ball packable, and in general, $K_3\star G_{k(d-1,1)+1}$ is not $d$-ball packable.  

We can generalize this idea as follows. A \emph{$(d,\alpha)$-kissing configuration} is a packing of unit balls touching $\alpha$ pairwise tangent unit balls. The \emph{$(d,\alpha)$-kissing number}~$k(d,\alpha)$ is the maximum number of balls in a $(d,\alpha)$-kissing configuration.  So the~$d$-kissing configuration discussed above is actually the $(d,1)$-kissing configuration, from where the notation $k(d,1)$ is derived.  Clearly, if $G$ is the tangency graph of a~$(d,\alpha)$-kissing configuration, $G\star K_1$ must be the graph of a~$(d,\alpha-1)$-kissing configuration, and $G\star K_{\alpha-1}$ must be the graph of a $d$-kissing configuration.  With a similar argument as before, we have
\begin{theorem}\label{thm:akissing}
  A graph in the form of $K_{2+\alpha}\star G$ is $d$-ball packable if and only if~$G$ is the tangency graph of a ${(d-1,\alpha)}$-kissing configuration.
\end{theorem}
To see this, just represent $K_{2+\alpha}$ by two half-spaces at distance $2$ apart and $\alpha$ pairwise tangent unit balls, then the other balls must form a $(d-1,\alpha)$-kissing configuration.  As a consequence, a graph in the form of $K_{2+\alpha}\star G_{k(d-1,\alpha)+1}$ is not $d$-ball packable. The following corollary follows from the fact that ${k(d,d)=2}$ for all $d>0$
\begin{corollary}\label{cor:ktree}
  A graph in the form of $K_{d+1}\star G_3$ is not $d$-ball packable.
\end{corollary}
We then see from Theorem~\ref{thm:ktree} that a $(d+1)$-tree is $d$-ball packable only if it is stacked $(d+1)$-polytopal.

A \emph{$(d,\cos\theta)$-spherical code}~\cite{conway1999} is a set of points on the unit $(d-1)$-sphere such that the spherical distance between any two points in the set is at least~$\theta$. We denote by~$A(d,\cos\theta)$ the maximal number of points in such a spherical code.  Spherical codes generalize kissing configurations. The minimal spherical distance corresponds to the tangency relation, and $A(d,\cos\theta)=k(d,1)$ if $\theta=\pi/3$. Corresponding to the tangency graph, the \emph{minimal-distance graph} of a spherical code takes the points as vertices and connects two vertices if the corresponding points attain the minimal spherical distance.  As noticed by Bannai and Sloane~\cite{bannai1981}*{Theorem 1}, the centers of unit balls in a $(d,\alpha)$-kissing configuration correspond to a $(d-\alpha+1,\frac{1}{\alpha+1})$-spherical code after rescaling.  Therefore:
\begin{corollary}\label{cor:sphcod}
  A graph in the form of $K_{2+\alpha}\star G$ is $(d+\alpha)$-ball packable if and only if $G$ is the minimal-distance graph of a $(d,\frac{1}{\alpha+1})$-spherical code.
\end{corollary}
\begin{table}[b!]\footnotesize
  \caption{Some known $(d,\frac{1}{\alpha+1})$-spherical codes for integer $\alpha$}
  \begin{center}
    \begin{tabular}{l*{3}{c}}
      spherical code 			& minimal distance graph	& $\alpha$  & $d$ \\
      \hline \hline
      $k$-orthoplicial prism		&$\lozenge_k\square K_2$			&$2$&$k+1$\\
      $k$-orthoplicial-pyramidal prism	&$(\lozenge_k\star K_1)\square K_2$	&$2$&$k+2$\\
      rectified $k$-orthoplex		&$L(\lozenge_k)$			&$1$&$k$\\
      augmented $k$-simplicial prism	&					&$k$&$k+1$\\
      \hline
      $2$-simplicial prism($-1_{21}$)~\cite{klitzing2000}*{3.4.1}&$K_3\square K_2$&$6$&$3$\\ 
      $3$-simplicial prism($-1_{31}$)~\cite{klitzing2000}*{4.9.2}&$K_4\square K_2$&$4$&$4$\\
      $5$-simplicial prism					&$K_6\square K_2$&$3$&$6$\\
      \hline
      triangle-triangle duoprism($-1_{22}$)\cite{klitzing2000}*{4.10}	&$K_3\square K_3$&$3$&$4$\\ 
      tetrahedron-tetrahedron duoprism				&$K_4\square K_4$&$2$&$6$\\
      triangle-hexahedron duoprism				&$K_3\square K_6$&$2$&$7$\\ 
      \hline
      rectified $4$-simplex($0_{21}$)~\cite{bachoc2009}&$J_{5,2}$	&$5$&$4$\\
      rectified $5$-simplex($0_{31}$)				&$J_{6,2}$	&$3$&$5$\\ 
      rectified $7$-simplex					&$J_{8,2}$	&$2$&$7$\\
      birectified $5$-simplex($0_{22}$)				&$J_{6,3}$	&$2$&$5$\\ 
      birectified $8$-simplex					&$J_{9,3}$	&$1$&$8$\\
      trirectified $7$-simplex					&$J_{8,4}$	&$1$&$7$\\
      \hline
      $5$-demicube($1_{21}$)~\cite{sloane}*{pack.5.16}		&		&$4$&$5$\\ 
      $6$-demicube($1_{31}$)					&		&$2$&$6$\\
      $8$-demicube						&		&$1$&$8$\\
      \hline
      $1_{22}$							&		&$1$&$6$\\ 
      $2_{31}$							&		&$1$&$7$\\ 
      $2_{21}$	\cite{cohn2007b}*{Appendix A}			&		&$3$&$6$\\ 
      $3_{21}$	\cite{bannai1981}				&		&$2$&$7$\\
      $4_{21}$	\cite{bannai1981}				&		&$1$&$8$\\
      \hline
      3p\textbardbl refl ortho 3p~\cite{klitzing2000}*{4.13}	&		&$2$&$4$\\
      3g\textbardbl gyro 3p~\cite{klitzing2000}*{4.6.2}		&		&$5$&$4$\\
      3g\textbardbl ortho 4g~\cite{klitzing2000}*{4.7.3}	&		&$5$&$4$\\
      3p\textbardbl ortho line~\cite{klitzing2000}*{4.8.2}	&		&$5$&$4$\\
      oct\textbardbl hex~\cite{sloane}*{pack.5.14}		&		&$4$&$5$\\
      \hline
    \end{tabular}
    \label{tab:sphcod}
  \end{center}
\end{table}
We give in Table~\ref{tab:sphcod} an incomplete list of $(d,\frac{1}{\alpha+1})$-spherical codes for integer values of $\alpha$. They are therefore $(d+\alpha-1,\alpha)$-kissing configurations for the $\alpha$ and $d$ given in the table.  The first column is the name of the polytope whose vertices form the spherical code. Some of them are from Klitzing's list of segmentochora~\cite{klitzing2000}, which can be viewed as a special type of spherical codes. Some others are inspired from Sloane's collection of optimal spherical codes~\cite{sloane}. For those polytopes with no conventional name, we keep Klitzing's notation, or give a name following Klitzing's method.  The second column is the corresponding minimal-distance graph, if a conventional notation is available. Here are some notations used in the table: 
\begin{itemize}
  \item For a graph $G$, its \emph{line graph} $L(G)$ takes the edges of $G$ as vertices, and two vertices are adjacent iff the corresponding edges share a vertex in $G$. 
  \item The \emph{Johnson graph} $J_{n,k}$ takes the $k$-element subsets of an $n$-element set as vertices, and two vertices are adjacent whenever their intersection contains $k-1$ elements. Especially, $J_{n,2}=L(K_n)$. 
  \item For two graph $G$ and $H$, $G\square H$ denotes the \emph{Cartesian product}.  
\end{itemize}
We would like to point out that for $1\leq\alpha\leq 6$, vertices of the uniform ${(5-\alpha)_{21}}$~polytope form an $(8,\alpha)$-kissing configuration. These codes are derived from the $E_8$ root lattice~\cite{bannai1981}*{Example 2}. They are optimal and unique except for the trigonal prism($(-1)_{21}$~polytope)~\citelist{\cite{cohn2007b}*{Appendix A}\cite{bachoc2009}}. There are also spherical codes similarly derived from the \emph{Leech lattice}~\citelist{\cite{bannai1981}*{Example 3}\cite{cohn2007}}.

As another example, since
\[k(d,\alpha)=A\Big(d-\alpha+1,\frac{1}{\alpha+1}\Big).\]
the following fact provides another proof Corollary~\ref{cor:gk}:
\[
k(d,d-1)=A(2,1/d)=\begin{cases}
  4 &\text{if $d\geq 4$}\\
  5 &\text{if $d=3$}\\
  6 &\text{if $d=2$(optimal)}
\end{cases}
\]

Before ending this part, we present the following trivial theorem.
\begin{theorem}
  A graph in the form of $K_2\star G$ is $d$-ball packable if and only if $G$ is~$(d-1)$-unit-ball packable.
\end{theorem}
For the proof, just use disjoint half-spaces to represent $K_2$, then $G$ must be representable by a packing of unit balls.

\subsection{Graphs in the form of $\lozenge_d\star G_m$}
\begin{theorem}\label{thm:odp4}
  A graph in the form of $\lozenge_{d-1}\star P_4$ is not $d$-ball packable, but $\lozenge_{d+1}=\lozenge_{d-1}\star C_4$ is.
\end{theorem}
\begin{proof}
  The graph $\lozenge_{d-1}$ is the $1$-skeleton of the $(d-1)$-dimensional orthoplex.  The vertices of a regular orthoplex of edge length $\sqrt 2$ forms an optimal spherical code of minimal distance $\pi/2$.  As in the proof of Theorem~\ref{thm:pk}, we first construct the ball packing of $\lozenge_{d-1}\star P_2$. The edge $P_2$ is represented by two disjoint half-spaces. The graph $\lozenge_{d-1}$ is represented by $2(d-1)$ unit balls. Their centers are on a $(d-2)$-dimensional sphere $\mathsf{S}$, otherwise further construction would not be possible. So the centers of these unit balls must be the vertices of a regular $(d-1)$-dimensional orthoplex of edge length $2$, and the radius of~$\mathsf{S}$ is $1/\sqrt 2$.  

  We now construct $\lozenge_{d-1}\star P_3$ by adding the unique ball that is tangent to all the unit balls and also to one half-space. After an elementary calculation, the radius of this ball is $1/2$. By symmetry, a ball touching the other half-space has the same radius. These two balls must be tangent since their diameters sum up to $2$.  Therefore, an attempt for constructing a ball packing of $\lozenge_{d-1}\star P_4$ results in a ball packing of $\lozenge_{d+1}=\lozenge_{d-1}\star C_4$.
\end{proof}

\begin{figure}[htb]
  \centering
  \includegraphics[width=\textwidth]{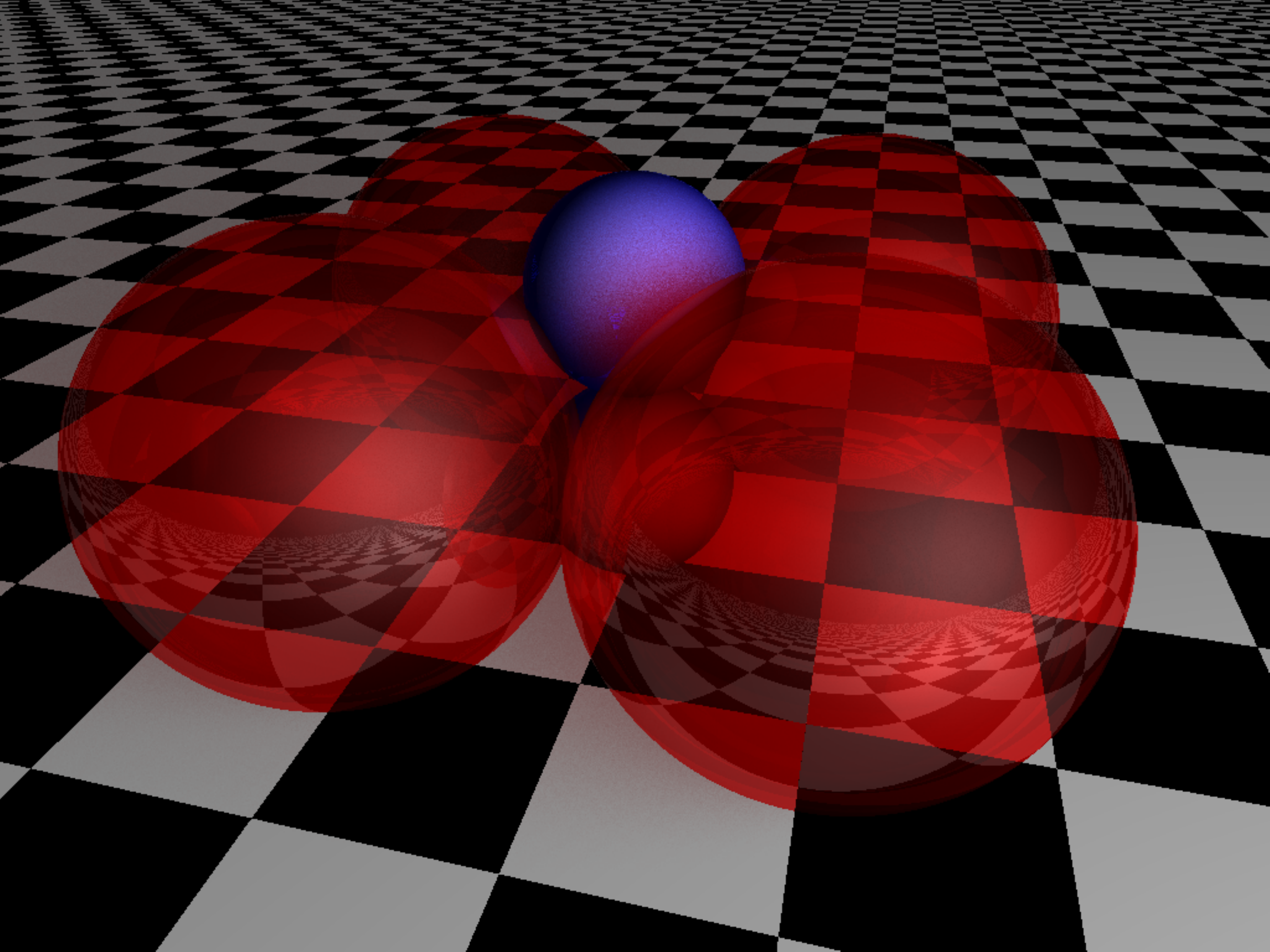}
  \caption{\label{pic:C4C4} A ball packing of $C_4\star C_4$. The red balls form a cycle, and the blue balls form a cycle with the lower and the upper half-space. The upper half-space is not shown. The image is rendered by POV-Ray.}
\end{figure}

For example, $C_4\star C_4$ is $3$-ball packable, as shown in Figure~\ref{pic:C4C4}. This is also observed by Maehara and Oshiro in~\cite{maehara2000}.  By the same argument as in the proof of Corollary~\ref{cor:gk}, we have 
\begin{corollary}\label{cor:odg4}
  A graph in the form of $\lozenge_{d-1}\star G_4$ is not $d$-ball packable, with the exception of $\lozenge_{d+1}=\lozenge_{d-1}\star C_4$.
\end{corollary}

\subsection{Graphs in the form of $G_n\star G_m$}

The following is a corollary of Corollary~\ref{cor:gk}.
\begin{corollary}\label{cor:g6g3}
  A graph in the form of $G_6\star G_3$ is not $3$-ball packable, with the exception of $C_6\star C_3$.
\end{corollary}
\begin{proof}
  As in the proof of Theorem~\ref{thm:pk}, up to M\"obius transformations, we may represent $G_3$ by three unit balls. We assume that their centers are not collinear, otherwise further construction is not possible. Let $\mathsf{S}$ be the $1$-sphere decided by their centers. Every ball representing a vertex of $G_6$ must center on the straight line passing through the center of $\mathsf{S}$ perpendicular to the plane containing $\mathsf{S}$. From the proof of Corollary~\ref{cor:gk}, the number of disjoint balls touching all three unit balls is at most six, while six balls only happens in the Soddy's hexlet.
\end{proof}

The following corollaries follow from the same argument with slight modification.
\begin{corollary}\label{cor:g4g4}
  A graph in the form of $G_4\star G_4$ is not $3$-ball packable, with the exception of $C_4\star C_4$.
\end{corollary}
\begin{proof}
  Up to M\"obius transformation, we may represent three vertices of the first~$G_4$ by unit balls, whose centers decide a $1$-sphere $\mathsf{S}$. Balls representing vertices of the second $G_4$ must center on the straight line passing through the center of $\mathsf{S}$ perpendicular to the plane containing $\mathsf{S}$. Then the remaining vertex of the first~$G_4$ must be represented by a unit ball centered on $\mathsf{S}$, too. We conclude from Corollary~\ref{cor:odg4} that the only possibility is $C_4\star C_4$.
\end{proof}

\begin{corollary}\label{cor:g4g6}
  A graph in the form of $G_4\star G_6$ is not $4$-ball packable, with the exception of $C_4\star\lozenge_3$.
\end{corollary}
\begin{proof}
  Up to M\"obius transformation, we represent four vertices of $G_6$ by unit balls, whose centers decide a $2$-sphere $\mathsf{S}$. Balls representing vertices of $G_4$ must center on the straight line passing through the center of $\mathsf{S}$ perpendicular to the hyperplane containing $\mathsf{S}$. Then the two remaining vertices of $G_6$ must be represented by unit balls centered on $\mathsf{S}$, too.  The diameter of $\mathsf{S}$ is minimal \emph{only} when $G_6=\lozenge_3$.  In this case, $G_4$ must be in the form of $C_4$ by Corollary~\ref{cor:odg4}. If $G_6$ is in any other form, a ball touching the unit balls must have a larger radius, which is not possible.  

  Special caution is needed for a degenerate case. It is possible to have the six unit balls centered on a $1$-sphere. In this case, the radius of a ball touching all of them is at least $1$, which rules out the possibility of further construction.
\end{proof}

Therefore, if a graph is $3$-ball packable, any induced subgraph in the form of $G_6\star G_3$ must be in the form of $C_6\star K_3$, and any induced subgraph in the form of $G_4\star G_4$ must be in the form of~$C_4\star C_4$. If a graph is $4$-ball packable, every induced subgraph in the form of $G_4\star G_6$ must be in the form of $C_4\star\lozenge_3$.

\begin{remark}
  The argument in these proofs should be used with caution. As mentioned in the proof of Corollary~\ref{cor:g4g6}, one must check carefully the degenerate cases. In higher dimensions, we are in general not so lucky, so these results does not generalize.
\end{remark}

The following is a corollary of Theorem~\ref{thm:akissing}, for which we omit the simple proof.

\begin{corollary}
  A graph in the form of $K_2\star G_\alpha\star G_{k(d-1,\alpha)+1}$ is not $d$-ball packable.
\end{corollary}

\section{Ball packable stacked-polytopal graphs}\label{sec:stacked}
This section is devoted to the proof of the main result. Some proof techniques are adapted from~\cite{graham2006}.

\subsection{More on stacked polytopes}
Since a graph in the form of $K_d\star P_m$ is stacked $(d+1)$-polytopal, Theorem~\ref{thm:pk} provides some examples of stacked $(d+1)$-polytope whose graph is not $d$-ball packable, and $C_3\star C_6$ provides an example of Apollonian $3$-ball packing whose tangency graph is not stacked $4$-polytopal. Therefore, in higher dimensions, the relation between Apollonian ball packings and stacked polytopes is more complicated.  The following remains true:

\begin{theorem}\label{thm:ApolUniq}
  If the graph of a stacked $(d+1)$-polytope is $d$-ball packable, its ball packing is Apollonian and unique up to M\"obius transformations and reflections.
\end{theorem}
\begin{proof}
  The Apollonianity can be easily seen by comparing the construction processes.  The uniqueness can be proved by an induction on the construction process.  While a stacked polytope is built from a simplex, we construct its ball packing from a Descarte configuration, which is unique up to M\"obius transformations and reflections.  For every stacking operation, a new ball representing the new vertex was added into the packing, forming a new Descartes configuration.  We have an unique choice for every newly added ball, so the uniqueness is preserved at every step of construction.
\end{proof}

For a $d$-polytope $\mathcal{P}$, the \emph{link} of a $k$-face $F$ is the subgraph of $G(\mathcal{P})$ induced by the common neighbors of the vertices of $F$.  The following lemma will be useful for the proofs later:
\begin{lemma}\label{lem:pathneighbor}
  For a stacked $d$-polytope $\mathcal{P}$, the link of a $k$-face is stacked $(d-k-1)$-polytopal.
\end{lemma}
\subsection{Weighted mass of a word}
The following theorem was proved in~\cite{graham2006}
\begin{theorem}\label{thm:graham}
  The $3$-dimensional Apollonian group is a hyperbolic Coxeter group generated by the relations $\R_i\R_i=\I$ and $(\R_i\R_j)^3=\I$ for $1\leq i\neq j\leq 5$.
\end{theorem}

Here we sketch the proof in~\cite{graham2006}, which is based on the study of reduced words.
\begin{definition}
  A word $\U=\U_1\U_2\cdots \U_n$ over the generator of the $3$-dimensional Apollonian group (i.e.~$\U_i\in\{\R_1,\cdots,\R_5\}$)
  is \emph{reduced} if it does not contain
  \begin{itemize}
    \item subword in the form of $\R_i\R_i$ for $1\leq i\leq 5$; or
    \item subword in the form of $\V_1\V_2\cdots \V_{2m}$ in which $\V_1=\V_3$, $\V_{2m-2}=\V_{2m}$ and $\V_{2j}=\V_{2j+3}$ for $1\leq j\leq 2m-2$.
  \end{itemize}
\end{definition}
Notice that $m=2$ excludes the subwords of the form $(\R_i\R_j)^2$.  One verifies that a non-reduced word can be simplified to a reduced word using the generating relations.  Then it suffices to prove that no nonempty reduced word, treated as product of matrices, is identity.

To prove this, the authors of~\cite{graham2006} studied the sum of entries in the $i$-th row of~$\U$, i.e.~$\sigma_i(\U):=\e_i^\intercal\U\e$, and the sum of all the entries in $\U$, i.e.~$\Sigma(\U):=\e^\intercal\U\e$. The latter is called the \emph{mass} of~$\U$.  The quantities $\Sigma(\U)$, $\Sigma(\R_j\U)$, $\sigma_i(\U)$ and $\sigma_i(\R_j\U)$ satisfy a series of linear equations, which was used to inductively prove that $\Sigma(\U)>\Sigma(\U')$ for a reduced word $\U=\R_i\U'$.  Therefore $\U$ is not an identity since $\Sigma(\U)\geq \Sigma(\R_i)=7>\Sigma(\I)=5$.

We propose the following adaption. Given a weight vector $\w$, we define $\sigma_i^w(\U)=\e_i^\intercal\U\w$ the weighted sum of entries in the $i$-th row of $\U$, and $\Sigma^w(\U)=\e^\intercal\U\w$ the \emph{weighted mass} of $\U$.  The following lemma can be proved with an argument similar as in~\cite{graham2006}:

\begin{lemma}\label{lem:incmass}
  For dimension $3$, if $\Sigma^w(\R_i)\geq\Sigma^w(\I)$ for any $1\leq i\leq 5$, then for a reduced word $\U=\R_i\U'$, we have $\Sigma^w(\U)\geq\Sigma^w(\U')$.
\end{lemma}
\begin{proof}[Sketch of proof]
  It suffices to replace ``sum'' by ``weighted sum'', ``mass'' by ``weighted mass'', and ``$>$'' by ``$\geq$'' in the proof of~\cite{graham2006}*{Theorem 5.1}.  It turns out that the following relations hold for $1\leq i,j\leq 5$.
  \begin{align}
    \sigma_i^w(\R_j\U) &= \begin{cases}
      \sigma^w_i(\U) & \text{if } i\neq j\\
      \Sigma^w(\U)-2\sigma^w_i(\U) & \text{if } i=j
    \end{cases}\label{eq:rowsum}\\
    \Sigma^w(\R_i\U) &= 2\Sigma^w(\U)-3\sigma^w_i(\U) \nonumber
  \end{align}
  Then, if we define $\delta^w_i(\U):=\Sigma^w(\R_i\U)-\Sigma^w(\U)$, the following relations hold:
  \begin{align*}
    \delta^w_i(\R_j\U) &= \begin{cases}
      \delta^w_i(\U)+\delta^w_j(\U) & \text{if } i\neq j\\
      -\delta^w_i(\U) & \text{if } i=j
    \end{cases}\\
    \delta^w_i(\R_j\U) &= \delta^w_j(\R_i\U) \text{ if } i\neq j\\
    \delta^w_i(\R_j\R_i\U) &= \delta^w_j(\U)
  \end{align*}
  These relations suffice for the induction.  The base case is already assumed in the condition of the theorem, which reads $\delta^w_i(\I)\geq 0$ for $1\leq i\leq 5$.  So the rest of the proof is exactly the same as in the proof of~\cite{graham2006}*{Theorem 5.1}.  For details of the induction, please refer to the original proof.  The conclusion is $\delta^w_i(\U')\geq 0$, i.e.~$\Sigma^w(\U)\geq\Sigma^w(\U')$.
\end{proof}

\subsection{A generalization of Coxeter's sequence}\label{sse:genseq}
Let $\U=\U_n\cdots \U_2\U_1$ be a word over the generators of the $3$-dimensional Apollonian group (we have a good reason for inversing the order of the index).  Let $\M_0$ be the curvature-center matrix of an initial Descartes configuration, consisting of five balls $S_1,\cdots,S_5$.  The curvature-center matrices recursively defined by $\M_i=\U_i\M_{i-1}$, $1\leq i\leq n$, define a sequence of Descartes configurations.  We take $S_{5+i}$ to be the single ball that is in the configuration at step $i$ but not in the configuration at step $i-1$.  This generates a sequence of $5+n$ balls, which generalizes  Coxeter's loxodromic sequence in dimension~$3$. In fact, Coxeter's loxodromic sequence is generated by an infinite word of period~$5$, e.g. $\U=\cdots \R_2\R_1\R_5\R_4\R_3\R_2\R_1$.

\begin{lemma}\label{lem:pathtree}
  If $\U$ is reduced and $\U_1=\R_1$, then in the sequence constructed above, $S_1$ is disjoint from every ball except the first five.
\end{lemma}
\begin{proof}
  We take the initial configuration to be the configuration used in the proof of Theorem~\ref{thm:pk}.  Assume $S_1$ to be the lower half-space $x_1\leq 0$, then the initial curvature-center matrix is
  \[ \M_0=
    \begin{pmatrix}
      0 & -1 & 0 & 0\\
      0 & 1 & 0 & 0\\
      1 & 1 & 1 & \sqrt{1/3}\\
      1 & 1 & -1 & \sqrt{1/3}\\
      1 & 1 & 0 & -2\sqrt{1/3}
    \end{pmatrix}
  \]
  Every row corresponds to the curvature-center coordinates $\m$ of a ball.  The first coordinate $m_1$ is the curvature $\kappa$.  If the curvature is not zero, the second coordinate $m_2$ is the ``height'' of the center times the curvature, i.e.~$x_1\kappa$.

  Now take the second column of $\M_0$ to be the weight vector $\w$.  That is, $$\w=(-1,1,1,1,1)^\intercal.$$ We have $\Sigma^w(\R_1)=9>\Sigma^w(\I)=3$ and $\Sigma^w(\R_j)=3=\Sigma^w(\I)$ for $j>1$.  By Lemma~\ref{lem:incmass}, we have
  \[ \Sigma^w(\U_k\U_{k-1}\cdots \U_2\R_1)\geq\Sigma^w(\U_{k-1}\cdots \U_2\R_1) \]
  By \eqref{eq:rowsum}, this means that
  \[ \sigma_j^w(\U_k\cdots \U_2\R_1)\geq\sigma_j^w(\U_{k-1}\cdots \U_2\R_1) \]
  if $\U_k=\R_j$, or equality if $\U_k\neq \R_j$.

  The key observation is that $\sigma_j^w(\U_k\cdots \R_1)$ is nothing but the second curvature-center coordinate $m_2$ of the $j$-th ball in the $k$-th Descartes configuration.  So at every step, a ball is replaced by another ball with a larger or same value for $m_2$.  Especially, since $\sigma_j^w(\R_1)\geq 1$ for $1\leq j\leq 5$, we conclude that $m_2\geq 1$ for every ball.

  Four balls in the initial configuration have $m_2=1$.  Once they are replaced, the new ball must have a \emph{strictly} larger value of $m_2$.  This can be seen from \eqref{eq:curvcenter} and notice that the r.h.s.~of~\eqref{eq:curvcenter} is at least $4$ since the very first step of the construction.  We then conclude that $m_2>1$ for all balls except the first five.  This exclude the possibility of curvature zero, so $x_1\kappa>1$ for all balls except the first five.

  For dimension $3$, Equation \eqref{eq:curvcenter} is integral.  Therefore the curvature-center coordinates of all balls are integral (see~\cite{graham2006} for more details on integrality of Apollonian packings).  Since the sequence is a packing (by the result of~\cite{boyd1973}), no ball in the sequence has a negative curvature.  By the definition of the curvature-center coordinates, the fact that $m_2>1$ exclude the possibility of curvature $0$.  Therefore all balls have a positive curvature $\kappa\geq 1$ except the first two.

  For conclusion, $x_1\kappa>1$ and $\kappa\geq 1$ implies that $x_1>1/\kappa$, therefore disjoint from the half-space~$x_1\leq 0$.
\end{proof}

\subsection{Main result}
\begin{lemma}\label{lem:k3p6}
  Let $G$ be a stacked $4$-polytopal graph.  If $G$ has an induced subgraph in the form of $G_3\star G_6$, then $G$ must have an induced subgraph in the form of $K_3\star P_6$.
\end{lemma}
Note that $C_6\star K_3$ is not an induced subgraph of any stacked polytopal graph.
\begin{proof}
  Let $H$ be an induced subgraph of $G$ of form $G_3\star G_6$.  Let $v\in V(H)$ be the last vertex of~$H$ that is added into the polytope during the construction of $G$.  We have $\deg_Hv=4$, and the neighbors of $v$ induce a complete graph.  So the vertex~$v$ must be a vertex of $G_6$. On the other hand, $G_3$ is an induced subgraph of $K_4$, therefore must be the complete graph $K_3$.  Hence $H$ is of the form $K_3\star G_6$.

  By Lemma~\ref{lem:pathneighbor}, in the stacked $4$-polytope with graph $G$, the link of every $2$-face is stacked $1$-polytopal.  In other words, the common neighbors of $K_3$ induce a path $P_n$ where $n\geq 6$.  Therefore $G$ must have an induced subgraph of the form $P_6\star K_3$.
\end{proof}

\begin{proof}[proof of Theorem~\ref{thm:main}]
  The ``only if'' follows from Theorem~\ref{thm:pk} and Lemma~\ref{lem:k3p6}.  We prove the ``if'' part by induction on number of vertices.

  The complete graph on $5$ vertices is clearly $3$-ball packable.  Assume that every stacked $4$-polytope with less than $n$ vertices satisfies this theorem.  We now study a stacked $4$-polytope $\mathcal{P}$ of $n+1$ vertices that do not have six $4$-cliques in its graph with $3$ vertices in common, and assume that $G(\mathcal{P})$ is not ball packable.

  Let $u,v$ be two vertices of $G(\mathcal{P})$ of degree $4$.  Deleting $v$ from $\mathcal{P}$ leaves a stacked polytope $\mathcal{P}'$ of $n$ vertices that satisfies the condition of the theorem, so $G(\mathcal{P}')$ is ball packable by the assumption of induction.  In the ball packing of $\mathcal{P}'$, the four balls corresponding to the neighbors of $v$ are pairwise tangent.  We then construct the ball packing of $\mathcal{P}$ by adding a ball $S_v$ that is tangent to these four balls.  We have only one choice (the other choice coincides with another ball), but since $G(\mathcal{P})$ is not ball packable, $S_v$ must intersect some other balls.

  However, deleting $u$ also leaves a stacked polytope whose graph is ball packable. Therefore $S_v$ must intersect $S_u$ and only $S_u$.  Now if there is another vertex $w$ of degree $4$ different from $u$ and $v$, deleting $w$ leaves a stacked polytope whose graph is ball packable, which produces a contradiction.  Therefore $u$ and $v$ are the only vertices of degree $4$.

  Let $\mathcal{T}$ be the dual tree of $\mathcal{P}$, its leaves correspond to vertices of degree $4$.  So $\mathcal{T}$ must be a path, whose two ends correspond to $u$ and $v$.  We can therefore construct the ball packing of $\mathcal{P}$ as a generalised Coxeter's sequence studied in the previous part.  The first ball is $S_u$.  The construction word does not contain any subword of form $(\R_i\R_j)^2$ (which produces $C_6\star K_3$ and violates the condition) or $\R_i\R_i$, one can therefore always simplify the word into a \emph{non-empty} reduced word. This does not change the corresponding matrix, so the curvature-center matrix of the last Descarte configuration remains the same.

  Then Lemma~\ref{lem:pathtree} says that $S_u$ and $S_v$ are disjoint, which contradicts our previous discussion.  Therefore $G(\mathcal{P})$ is ball packable.
\end{proof}
\begin{corollary}[of the proof]\label{cor:hexletfree}
  The tangency graph of an Apollonian $3$-ball packing is a $4$-tree if and only if it does not contain any Soddy's hexlet.
\end{corollary}
\begin{proof}
  The ``only if'' part is trivial. We only need to proof the ``if'' part.

  If the tangency graph is a $4$-tree, then during the construction, every newly added ball touches exactly $4$ pairwise tangent balls.  If it is not the case, we can assume $S$ to be the first ball that touches five balls, the extra ball being $S'$.

  Since the tangency graph is stacked $4$-polytopal before introducing $S$, there is a sequence of Descartes configurations generated by a word, with $S'$ in the first configuration and $S$ in the last one.  By ignoring the leading configurations in the sequence if necessary, we may assume that the second Descartes configuration does not contain $S'$.  We can arrange the first configuration as in the previous proof, taking $S'$ as the lower half-space $x_1<0$ and labelling it as the first ball. Therefore the generating word $\U$ ends with $\R_1$. 

  We may assume that $\U$ does not have any subword of the form $\R_i\R_i$. If~$\U$ is reduced, we know in the proof of Theorem~\ref{thm:main} that $S$ and $S'$ are disjoint, contradiction. So $\U$ is non-reduced, but we may simplify $\U$ to a reduced one~$\U'$. This will not change the curvature-center matrix of the last Descartes configuration. After this simplification, the last letter of $\U'$ can not be $\R_1$ anymore, otherwise $S$ and $S'$ are disjoint. If $\U$ ends with $\R_i\R_1$, then $\U'$ ends with $\R_1\R_i$.  
  
  In the sequence of balls generated by $\U'$, the only ball that touches $S'$ but not in the initial Descartes configuration is generated at the first step by $\R_i$.  This ball must be $S$ by assumption.  This is the only occurrence of $\R_i$ in $\U'$, otherwise $S$ is not contained in the last Descartes configuration generated by $\U'$.  Since $S$ is the last ball generated by $\U$, $\R_i$ must be the first letter of $\U$.  The only possibility is then $\U=\R_i\R_1\R_i\R_1$, which implies the presence of Soddy's hexlet.
\end{proof}

Therefore, the relation between $4$-trees, stacked $4$-polytopes and Apollonian $3$-ball packings can be illustrated as follows:

\begin{center}
  \begin{tikzpicture}[node distance=3cm]
    \node (tree) {$4$-tree};
    \node (polytope) [below left=of tree, xshift=1cm] {stacked $4$-polytope};
    \node (packing) [below right=of tree, xshift=-1cm] {Apollonian $3$-ball packing};
    \draw [->>] (tree) to node [sloped,above] {no three $5$-cliques} node [sloped,below] {sharing a $4$-clique} (polytope);
    \draw [right hook->] (polytope) to node [sloped,above] {no six $4$-cliques} node [sloped,below] {sharing a $3$-clique} (packing);
    \draw [->] (packing) to node [sloped,above] {does not contain} node [sloped,below] {Soddy's hexlet} (tree);
  \end{tikzpicture}
\end{center}
where the hooked arrow $A\hookrightarrow B$ means that every instance of $A$ corresponds to an instance of $B$ satisfying the given condition. 

\subsection{Higher dimensions}
In dimensions higher than $3$, the following relation between Apollonian packing and stacked polytope is restored.
\begin{theorem}\label{thm:not3d}
  For $d>3$, if a $d$-ball packing is Apollonian, then its tangency graph is stacked $(d+1)$-polytopal.
\end{theorem}

We will need the following lemma:
\begin{lemma}\label{lem:not3d}
  If $d\neq 3$, let $\w$ be the $(d+2)$ dimensional vector $(-1,1,\dots,1)^\intercal$, and $\U=\U_n\dots\U_2\U_1$ be a word over the generators of the $d$-dimensional Apollonian group (i.e.~$\U_i\in\{\R_1,\cdots,\R_{d+2}\}$).  If $\U$ ends with $\R_1$ and does not contain any subword of the form $\R_i\R_i$, then $\sigma_i^w(\U)\neq 1$ for $1\leq i\leq d+2$ as long as $\U$ contains the letter $\R_i$.
\end{lemma}
\begin{proof}
  It is shown in~\cite{graham2006}*{Theorem 5.2} that the $j$-th row of $\U-\I$ is a linear combination of rows of the matrix $\A=\frac{1}{d-1}\e\e^\intercal-d\I$. However, the weighted row sum $\sigma_i^w(\A)$ of the $i$-th row of $\A$ is $0$ except for $i=1$, whose weighted row sum is $\frac{2}{d-1}$. So $\sigma_i^w(\U-\I)=\frac{2C_i}{d-1}$, where $C_i$ is the coefficient in the linear combination.

  According to the calculation in~\cite{graham2006}, $C_i$ is a polynomial in the variable $x_d=\frac{1}{d-1}$ in the form of \[C_i(x_d)=\sum_{k=0}^{n_i-1} c_k2^{k+1}x_d^k\] where $n_i$ is the length of the longest subword that starts with $\R_i$ and ends with $\R_1$, and $c_k$ are integer coefficients. The leading term is $2^{n_i}x_d^{n_i-1}$ (i.e.~$c_{n_i-1}=1$).  Then, by the same argument as in~\cite{graham2006}, we can show that $C_i(x_d)$ is not zero as long as $\U$ contains $\R_i$.  Therefore, for $i\neq 1$, \[\sigma_i^w(\U)=\frac{2C_i}{d-1}+\sigma_i^w(\I)=\frac{2C_i}{d-1}+1\neq 1.\]

  For $i=1$, since $\sigma_1^w(\I)=-1$, what we need to prove is that $C_1\neq d-1$. So the calculation is slightly different.  If $C_1= d-1$, then $x_d$ is a root of the polynomial $x_dC_1(x_d)-1$, whose leading term is $(2x_d)^{n_1}$.  By the rational root theorem, $d-1$ divides $2^{n_1}$. So we must have $d-1=2^p$ for some $p>1$, that is, $x_d=2^{-p}$. We then have \[\sum_{k=1}^{n_1}c_{k-1}2^{k(1-p)}=1.\] Multiply both side by $2^{(p-1)n_1}$, we got \[\sum_{k=1}^{n_1}c_{k-1}2^{(p-1)(n_1-k)}=2^{(p-1)n_1}.\] The right hand side is even since $(p-1)n_1>0$. The terms in the summation are even except for the last one since $(p-1)(n_1-k)>0$.  The last term in the summation is $c_{n_1-1}2^0=1$, so the left hand side is odd, which is the desired contradiction.  Therefore \[\sigma_1^w(\U)=\frac{2C_1}{d-1}+\sigma_1^w(\I)\neq 1.\]
\end{proof}

\begin{proof}[proof of Theorem~\ref{thm:not3d}]
  Consider a construction process of the Apollonian ball packing. The theorem is true at the first step. Assume that it remains true before the introduction of a ball $S$. We are going to prove that, once added, $S$ touches exactly $d+1$ pairwise tangent balls in the packing.

  If this is not the case, assume that $S$ touches a $(d+2)$-th ball $S'$, then we can find a sequence of Descartes configurations, with $S'$ in the first configuration and $S$ in the last, generated (similar as in Section~\ref{sse:genseq}) by a word over the generators of the $d$-dimensional Apollonian group with distinct adjacent terms.  Without loss of generality, we assume $S'$ to be the lower half-space $x_1\leq 0$, as in the proof of the Corollary~\ref{lem:pathtree}. Then Lemma~\ref{lem:not3d} says that no ball (except for the first $d+2$ balls) in this sequence is tangent to $S'$, contradicting our assumption.

  By induction, every newly added ball touches exactly $d+1$ pairwise tangent balls, so the tangency graph is a $(d+1)$-tree, and therefore $(d+1)$-polytopal.
\end{proof}

So the relation between $(d+1)$-trees, stacked $(d+1)$-polytopes and Apollonian $d$-ball packings can be illustrated as follows:
\begin{center}
  \begin{tikzpicture}[node distance=3.5cm]
    \node (tree) {$(d+1)$-tree};
    \node (polytope) [below left=of tree, xshift=1.5cm] {stacked $(d+1)$-polytope};
    \node (packing) [below right=of tree, xshift=-1.5cm] {Apollonian $d$-ball packing};
    \draw [->>] (tree) to node [sloped,above] {no three $(d+2)$-cliques} node [sloped,below] {sharing a $(d+1)$-clique} (polytope);
    \draw [->>] (polytope) to node [sloped,above] {unknown condition} (packing);
    \draw [right hook->] (packing) to (tree);
  \end{tikzpicture}
\end{center}

Now the remaining problem is to characterise stacked $(d+1)$-polytopal graphs that are $d$-ball packable.  From Corollary~\ref{cor:ktree}, we know that if a $(d+1)$-tree is $d$-ball packable, the number of $(\alpha+3)$-cliques sharing a $(\alpha+2)$-clique is at most $k(d-1,\alpha)$ for all $1\leq\alpha\leq d-1$.  Following the patterns in Theorems~\ref{thm:main} and~\ref{thm:ktree}, we propose the following conjecture:
\begin{conjecture}\label{conj:higher}
  For an integer $d\geq 2$, there is $d-1$ integers $n_1,\dots,n_{d-1}$ such that a $(d+1)$-tree is $d$-ball packable if and only if the number of $(\alpha+3)$-cliques sharing an $(\alpha+2)$-clique is at most $n_\alpha$ for all $1\leq\alpha\leq d-1$. 
\end{conjecture}

\section{Discussions}\label{sec:discuss}
A convex $(d+1)$-polytope is \emph{edge-tangent} if all of its edges are tangent to a $d$-sphere called \emph{midsphere}.  One can derive from the disk packing theorem that\footnotemark:
\begin{theorem}\label{thm:edgetangent}
Every convex $3$-polytope has an edge-tangent realization. 
\end{theorem}
\footnotetext{Schramm~\cite{schramm1992} said that the theorem is first claimed by Koebe~\cite{koebe1936}, who only proved the simplicial and simple cases. He credits the full proof to Thurston~\cite{thurston1979}, but the online version of Thurston's lecture notes only gave a proof for simplicial cases.}

Eppstein, Kuperberg and Ziegler have proved in~\cite{eppstein2003} that no stacked $4$-polytopes with more than six vertices has an edge-tangent realization. Comparing to Theorem~\ref{thm:main}, we see that ball packings and edge-tangent polytopes are not so closely related in higher dimensions: a polytope with ball packable graph does not, in general, have an edge-tangent realization.  In this part, we would like to discuss about this difference in detail.

\subsection{From ball packings to polytopes}
Let $\mathbb{S}^d\subset\mathbb{R}^{d+1}$ be the unit sphere $\{\x\mid x_0^2+\cdots+x_d^2=1\}$.  For a spherical cap $C\subset\mathbb{S}^d$ of radius smaller than $\pi/2$, its boundary can be viewed as the intersection of $\mathbb{S}^d$ with a $d$-dimensional hyperplane $H$, which can be uniquely written in form of $H=\{\x\in\mathbb{R}^d\mid\langle\x,\v\rangle=1\}$.  Explicitly, if $\c\in\mathbb{S}^d$ is the center of $C$, and $\theta<\pi/2$ is its spherical radius, then $\v=\c/\cos\theta$. We can interpret $\v$ as the center of the unique sphere that intersects $\mathbb{S}^d$ \emph{orthogonally} along the boundary of $C$, or as the apex of the unique cone whose boundary is tangent to $\mathbb{S}^d$ along the boundary of $C$. We call $\v$ the \emph{polar vertex of} $C$, and $H$ the \emph{hyperplane of} $C$.  We see that $\langle\v,\v\rangle>1$.  If the boundary of two caps $C$ and $C'$ intersect orthogonally, their polar vertices $\v$ and $\v'$ satisfy $\langle\v,\v'\rangle=1$, i.e. the polar vertex of one is on the hyperplane of the other.  If $C$ and $C'$ have disjoint interiors, $\langle\v,\v'\rangle<1$.  If $C$ and $C'$ are tangent at $\t\in\mathbb{S}^d$, the segment $\v\v'$ is tangent to $\mathbb{S}^d$ at $\t$.

Now, given a $d$-ball packing $\mathcal{S}=\{S_0,\cdots,S_n\}$ in $\hat{\mathbb{R}}^d$, we can construct a $(d+1)$-polytope $\mathcal{P}$ as follows.  View $\hat{\mathbb{R}}^d$ as the hyperplane $x_0=0$ in $\hat{\mathbb{R}}^{d+1}$. Then a stereographic projection maps $\hat{\mathbb{R}}^d$ to $\mathbb{S}^d$, and $\mathcal{S}$ is mapped to a packing of spherical caps on $\mathbb{S}^d$. With a M\"obius transformation if necessary, we may assume that the radii of all caps are smaller than $\pi/2$. Then $\mathcal{P}$ is obtained by taking the convex hull of the polar vertices of the spherical caps.

\begin{theorem}
  If a $(d+1)$-polytope $\mathcal{P}$ is constructed as described above from a $d$-sphere packing $\mathcal{S}$, then $G(\mathcal{S})$ is isomorphic to a spanning subgraph of $G(\mathcal{P})$.
\end{theorem}
\begin{proof}
  For every $S_i\in\mathcal{S}$, the polar vertex $\v_i$ of the corresponding cap is a vertex of $\mathcal{P}$, since the hyperplane $\{\x\mid\langle\v_i,\x\rangle=1\}$ divides $\v_i$ from other vertices.

  For every edge $S_iS_j$ of $G(\mathcal{S})$, we now prove that $\v_i\v_j$ is an edge of $\mathcal{P}$. Since $\v_i\v_j$ is tangent to the unit sphere, $\langle\v,\v\rangle\geq 1$ for all points $\v$ on the segment $\v_i\v_j$.  If $\v_i\v_j$ is not an edge of $\mathcal{P}$, some point $\v=\lambda\v_i+(1-\lambda)\v_j$ ($0\leq\lambda\leq 1)$ can be written as a convex combination of other vertices $\v=\sum_{k\neq i,j}\lambda_k\v_k$, where $\lambda_k\geq 0$ and $\sum\lambda_k\leq 1$. Then we have
  \[
  1\leq \langle\v,\v\rangle
  =\Big\langle\lambda\v_i+(1-\lambda)\v_j,\sum_{k\neq i,j}\lambda_k\v_k\Big\rangle<1
\]
  because $\langle\v_i,\v_j\rangle<1$ if $i\neq j$.  This is a contradiction.
\end{proof}

For an arbitrary $d$-ball packing $\mathcal{S}$, if a polytope $\mathcal{P}$ is constructed from $\mathcal{S}$ as described above, it is possible that $G(\mathcal{P})$ is not isomorphic to $G(\mathcal{S})$. That is, there may be an edge of $\mathcal{P}$ that does not correspond to any edge of $G(\mathcal{S})$. This edge will intersect $\mathbb{S}^d$, and $\mathcal{P}$ is therefore not edge-tangent.  On the other hand, if the graph of a polytope $\mathcal{P}$ is isomorphic to $G(\mathcal{S})$, since the graph does not determine the combinatorial type of a polytope, $\mathcal{P}$ may be different from the one constructed from $\mathcal{S}$. So a polytope whose graph is ball packable may not be edge-tangent.

\subsection{Edge-tangent polytopes}
A polytope is edge-tangent if it can be constructed from a ball packing as described above, and its graph is isomorphic to the tangency relation of this ball packing. Neither condition can be removed.
For the other direction, given an edge-tangent polytope $\mathcal{P}$, one can always obtain a ball packing of $G(\mathcal{P})$ by reversing the construction above. 

Disk packings are excepted from these problems. In fact, it is easier~\cite{sachs1994} to derive Theorem~\ref{thm:edgetangent} from the following version of the disk packing theorem, which is equivalent but contains more information: 

\begin{theorem}[Brightwell and Scheinerman~\cite{brightwell1993}]\label{thm:orthodisk} 
  For every $3$-polytope $\mathcal{P}$, there is a pair of disk packings, one consists of \emph{vertex-disks} representing $G(\mathcal{P})$, the other consists of \emph{face-disks} representing the dual graph $G(\mathcal{P}^*)$, such that:
  \begin{itemize}
    \item For each edge $e$ of $\mathcal{P}$, the vertex-disks corresponding to the two endpoints of $e$ and the face-disks corresponding to the two faces bounded by $e$ meet at a same point;
    \item A vertex-disk and a face-disk intersect iff the corresponding vertex is on the boundary of the corresponding face, in which case their boundaries intersect orthogonally.
  \end{itemize}
  This representation is unique up to M\"obius transformations.
\end{theorem}

The presence of the face-disks and the orthogonal intersections guarantee the incidence relations between vertices and faces, and therefore fix the combinatorial type of the polytope.  We can generalize this statement into higher dimensions:

\begin{theorem}\label{thm:orthoball}
  Given a $(d+1)$-polytope $\mathcal{P}$, if there is a packing of $d$-dimensional \emph{vertex-balls} representing $G(\mathcal{P})$, together with a collection of $(d-1)$-dimensional \emph{facet-balls} indexed by the facets of $\mathcal{P}$, such that:
  \begin{itemize}
    \item For each edge $e$ of $\mathcal{P}$, the vertex-balls corresponding to the two endpoints of $e$ and the boundaries of the facet-balls corresponding to the facets bounded by $e$ meet at a same point;
    \item Either a vertex-ball and a facet-ball are disjoint, or their boundaries intersect at a non-obtuse angle;
    \item The boundary of a vertex-ball and the boundary of a facet-ball intersect \emph{orthogonally} iff the corresponding vertex is on the boundary of the corresponding facet.
  \end{itemize}
  Then $\mathcal{P}$ has an edge-tangent realization.
\end{theorem}
Again, the convexity is guaranteed by the disjointness and nonobtuse intersections, and the incidence relations are guaranteed by the orthogonal intersections.  For an edge-tangent polytope, the facet-balls can be obtained by intersecting the midsphere with the facets.  However, they do not form a $d$-ball packing for $d>2$.  On the other hand, for an arbitrary polytope of dimension $4$ or higher, even if its graph is ball packable, the facet-balls satisfying the conditions of Theorem~\ref{thm:orthoball} do not in general exist. 

For example, consider the stacked $4$-polytope with $7$ vertices. The packing of its graph (with the form $K_3\star P_4$) is constructed in the proof of Theorem~\ref{thm:pk}. We notice that a ball whose boundary orthogonally intersects the boundary of the three unit balls and the boundary of ball $\mathsf{C}$, have to intersect the boundary of ball~$\mathsf{D}$ orthogonally (see Figure~\ref{pic:C6K3}), thus violates the last condition of Theorem~\ref{thm:orthoball}.  One verifies that the polytope constructed from this packing is not simplicial.

\subsection{Stress freeness}
Given a ball packing $\mathcal{S}=\{S_1,\cdots,S_n\}$, let $\v_i$ be the vertices of the polytope $\mathcal{P}$ constructed as above. A \emph{stress} of $\mathcal{S}$ is a real function $T$ on the edge set of $G(\mathcal{S})$ such that for all $S_i\in\mathcal{S}$ \[\sum_{S_iS_j \text{ edge of } G(\mathcal{S})}T(S_iS_j)(\v_j-\v_i)=0\] We can view stress as forces between tangent spherical caps when all caps are in equilibrium. We say that $\mathcal{S}$ is \emph{stress-free} if it has no non-zero stress.

\begin{theorem}
  If the graph of a stacked $(d+1)$-polytope is $d$-ball packable, its ball packing is stress-free.
\end{theorem}
\begin{proof}
  We construct the ball packing as we did in the proof of Theorem~\ref{thm:ApolUniq}, and assume a non-zero stress.  The last ball $S$ that is added into the packing has $d+1$ ``neighbor'' balls tangent to it. Let $\v$ be the vertex of $\mathcal{P}$ corresponding to $S$, and $C$ the correponding spherical caps on $\mathbb{S}^d$. If the stress is not zero on all the $d+1$ edges incident to $\v$, since $\mathcal{P}$ is convex, they can not be of the same sign. So there must be a hyperplane containing $\v$ separating positive edges and negative edges of $\v$. This contradicts the assumption that the spherical cap corresponding to $\v$ is in equilibrium. So the stress must vanish on the edges incident to $\v$. We then remove $S$ and repeat the same argument on the second last ball, and so on, and finally conclude that the stress has to be zero on all the edges of $G(\mathcal{S})$.
\end{proof}

The above theorem, as well as the proof, was informally discussed in Kotlov, Lov\'asz and Vempala's paper on Colin de Verdi\`ere number~\cite{kotlov1997}*{Section~8}.  In that paper, the authors defined an graph invariant $\nu(G)$ using the notion of stress-freeness, which turns out to be strongly related to Colin de Verdi\`ere number.  Their results imply that if the graph $G$ of a stacked $(d+1)$-polytope with $n$ vertices is $d$-ball packable, then $\nu(G)\leq d+2$, and the upper bound is achieved if $n\geq d+4$. However, Theorem~\ref{thm:pk} asserts that graphs of stacked polytopes are in general not ball packable.

\section*{Acknowledgement}
I'd like to thank Fernando M\'ario de Oliveira Filho, Jean-Philippe Labb\'e, Bernd Gonska and Prof. G\"unter M. Ziegler for helpful discussions.

\bibliography{References}

\end{document}